\documentclass[a4paper]{elsarticle}

\makeatletter
\def\ps@pprintTitle{%
\let\@oddhead\@empty
\let\@evenhead\@empty
\def\@oddfoot{}%
\let\@evenfoot\@oddfoot}
\makeatother

\usepackage[utf8]{inputenc}

\usepackage{amsmath,amsthm,amssymb}
\usepackage{geometry}
\usepackage{color}
\usepackage{hyperref}
\usepackage{comment}
\usepackage{mathdots}
\usepackage{mathtools}
\newtheorem{theorem}{Theorem}
\newtheorem{lemma}[theorem]{Lemma}
\newtheorem{remark}{Remark}

\newcommand{\CC}{\mathbb C}

\newcommand{\la}{\lambda}

\newcommand{\rank}{{\rm rank\,}}

\newcommand{\orb}{{\cal O}}
\newcommand{\codim}{{\rm codim\,}}
\newcommand{\maj}{{\prec_{\rm w}}}

\usepackage{tcolorbox}
\tcbuselibrary{skins}
\tcbuselibrary{breakable}
\tcbset{shield externalize}

\makeatletter
\def\namedlabel#1#2{\begingroup
    #2%
    \def\@currentlabel{#2}%
    \phantomsection\label{#1}\endgroup
}

\makeatletter
\def\namedlabel#1#2{\begingroup
    #2%
    \def\@currentlabel{#2}%
    \phantomsection\label{#1}\endgroup
}

\begin{document}

\title{On the dimension of orbits of matrix pencils \\ under strict equivalence}

\author[label1]{Fernando De Ter\'an}\author[label1]{Froil\'an M. Dopico}\author[label3]{Patryk Pagacz}

\address[label1]{Universidad Carlos III de Madrid, ROR: https://ror.org/03ths8210, Departamento de Matemáticas, Avenida de la Universidad, 30, 28911 Leganés (Madrid), Spain (fteran@math.uc3m.es, dopico@math.uc3m.es)}

\address[label3]{Jagiellonian University Krak\'{o}w, Prof. S. Lojasiewicza 6, Krak\'ow, 30-348, Krak\'{o}w, Poland (patryk.pagacz@gmail.com)}

\begin{abstract}
  We prove that, given two matrix pencils $L$ and $M$, if $M$ belongs to the closure of the orbit of $L$ under strict equivalence, then the dimension of the orbit of $M$ is smaller than or equal to the dimension of the orbit of $L$, and the equality is only attained when $M$ belongs to the orbit of $L$. Our proof uses only the majorization involving the eigenstructures of $L$ and $M$ which characterizes the inclusion relationship between orbit closures, together with the formula for the codimension of the orbit of a pencil in terms of its eigenstruture.
\end{abstract}


\maketitle

\section{Introduction}

The {\em orbit} (under strict equivalence) of a matrix pencil $A+\la B$, with $A,B\in\CC^{m\times n}$, is the set
$$
\orb(A+\la B)=\{P(A+\la B)Q:\ P\in\CC^{m\times m}\quad\mbox{and}\quad Q\in\CC^{n\times n}\quad\mbox{are invertible}\}.
$$
In other words, it is the set of strictly equivalent pencils to $A+\la B$ (where two pencils are {\em strictly equivalent} if one of them can be obtained from the other after multiplying on the left and on the right by invertible matrices). Strict equivalence preserves the {\em eigenstructure} of matrix pencils and this eigenstructure is revealed by the canonical form under strict equivalence, namely the {\em Kronecker canonical form} (KCF) (see \cite[Ch. XII]{Gant59}). Therefore, the orbit of $A+\la B$ is the set of all $m\times n$ pencils with the same eigenstructure as $A+\la B$. This makes orbits relevant tools in the theory and applications of matrix pencils, since the eigenstructure is one of the most important pieces of information about matrix pencils in the applications where they arise.

Orbits under strict equivalence of pencils are differentiable manifolds, as they are orbits under the action of a group (namely, $GL_m(\CC)\times GL_n(\CC)$) on a set (the set of $m\times n$ matrix pencils, which can be identified with $\CC^{m\times n}\times\CC^{m\times n}$). Then, the {\em dimension} of an orbit is the dimension of the tangent space at any point of the orbit.  Formulas for the dimension of orbits were obtained in \cite{DeEd95} from an explicit description of this tangent space. They have been reformulated in \cite{ddkp}.

One interesting topic in the study of orbits of matrix pencils is the inclusion relationship between their closures. More precisely, to determine whether the orbit of a given pencil, $M$, is included in the closure of the orbit of another pencil, $L$ (the closure is considered in the Euclidean topology of $\CC^{m\times n}\times\CC^{m\times n}$, and denoted by $\overline\orb$). A characterization for this inclusion relation was presented in \cite{Pokr86} (and later reformulated in subsequent works, like \cite{Hoyo90}) in terms of majorizations between some lists which encode the eigenstructure of $L$ and $M$ (see Section \ref{basic_sec}).

When $\overline\orb(M)\subseteq\overline\orb(L)$, then $\orb(L)$ is expected to be ``larger" than $\orb(M)$ namely $\dim\orb(M)\leq\dim\orb(L)$. This is actually true and well-known
(see the {\em Closed orbit lemma} in \cite[p. 53]{Bore91}). However, no direct proof of this inequality is known. The main goal of this paper is to derive the inequality directly from the characterization of the inclusion relationship between orbit closures and the known formulas for their dimension.

\section{Basic notions and results}\label{basic_sec}

Matrix pencils (or {\em pencils} for short) will be denoted by capital letters, like $L$ and $M$ (we drop the reference to the variable $\la$ for brevity). We also use the notation $\overline\CC:=\CC\cup\{\infty\}$.

 The KCF is a direct sum (namely, a block diagonal pencil), uniquely determined up to permutation of the blocks, of pencils (the diagonal blocks) of the following types:
(1) {\em Jordan blocks} of size $k\times k$, with $k\geq1$, associated with either finite or infinite eigenvalues, denoted by $J_k(\mu)$, where $\mu\in\overline\CC$ is the eigenvalue; (2) {\em Right singular blocks} with size $k\times(k+1)$, with $k\geq0$, denoted by $L_k$; and (3) {\em Left singular blocks} with size $(k+1)\times k$, with $k\geq0$, denoted by $L_k^\top$ (since they are transposes of right singular blocks). The form of these pencils is not relevant in this note (we refer to \cite[Ch. XII]{Gant59} for the details).

Given two lists of  non-increasingly ordered integers ${\cal M}=(m_1,m_2,\hdots)$ and ${\cal N}=(n_1,n_2,\hdots)$, we say that $\cal N$ {\em weakly majorizes} $\cal M$, denoted ${\cal M}\maj\, {\cal N}$, if $m_1+\cdots+ m_j\leq n_1+\cdots+n_j$, for $j\geq1$.

The {\em Weyr characteristic} of a finite  list, ${\cal S}$, of non-negative integers is the ordered list $ W({\cal S})=(W_0({\cal S}),W_1({\cal S}),W_2({\cal S}),\hdots)$, where $W_i({\cal S})$ is the number of integers in $\cal S$ which are greater than or equal to $i$ (in particular, $W_0({\cal S})$ is the length of $\cal S$). If ${\cal S}$ only contains positive integers, then its Weyr characteristic is $ W({\cal S})=(W_1({\cal S}),W_2({\cal S}),\hdots)$. The Weyr characteristics of the sizes of the Jordan blocks associated with $\mu\in\overline\CC$, the sizes of the right singular blocks, and the sizes of the left singular blocks in the KCF of some pencil $L$ are denoted, respectively, by $W(\mu,L)=(W_1(\mu,L),W_2(\mu,L),\hdots),r(L)=(r_0(L),r_1(L),\hdots)$, and $\ell(L)=(\ell_0(L),\ell_1(L),\hdots)$. Note that $W(\mu , L), r(L)$, and $\ell (L)$ have a finite number of non-zero elements. We can consider them as finite lists or infinite lists with all their terms equal to zero from a certain index on.

 We highlight that, for every $m\times n$ matrix pencil $L$, the following relation holds
\begin{equation}\label{nrank}
    m-\ell_0(L)=n-r_0(L):=\rank (L),
\end{equation}
where $\rank(L)$ denotes the {\em rank} of $L$ (see, for instance, \cite[Ch. XII, \S3]{Gant59}).

The following two identities, which hold for every $m\times n$ matrix pencil $M$, will be used later:
\begin{eqnarray}
    m=\sum_{i\geq1}r_i(M)+\sum_{i\geq0}\ell_i(M)+\sum_{\la\in\overline\CC}\sum_{i\geq1}W_i(\la,M),\label{m+p}\\
    n=\sum_{i\geq0}r_i(M)+\sum_{i\geq1}\ell_i(M)+\sum_{\la\in\overline\CC}\sum_{i\geq1}W_i(\la,M).\label{n+q}
\end{eqnarray}

A formula for the codimension of the orbit of an $m\times n$ pencil $L$ was obtained in \cite[Th. 2.2]{DeEd95}, in terms of the sizes of the blocks in its KCF. Recently, the following formula has been presented in \cite{ddkp} in terms of the Weyr characteristics:
\begin{equation}\label{codimorbit}
\codim\orb(L)=\displaystyle \ell_0(L)n+r_0(L)m-\displaystyle\sum_{i\geq0}{r_i(L)r_{i+1}(L)}-\sum_{i\geq0}{\ell_i(L)\ell_{i+1}(L)}+\sum_{\la\in\overline\CC}\sum_{i\geq1}{W_i(\lambda,L)^2}.
\end{equation}
The last term in \eqref{codimorbit} can be replaced with $\sum_{k=1}^p\sum_{i\geq1}{W_i(\lambda_k,L)^2}$, where $\la_1,\hdots,\la_p \in \overline\CC$ are the distinct eigenvalues of $L$, since $W_i(\la,L)=0$, for all $i\geq1$, if $\la$ is not an eigenvalue of $L$.

It is well-known (see the {\em Closed orbit lemma} in \cite[p. 53]{Bore91}) that, if $\overline\orb(M)\subseteq\overline\orb(L)$, then  $\dim {\cal O}(M)\leq \dim {\cal O}(L)$ (or, equivalently, $\codim\orb(L)\leq\codim\orb(M)$). The proof of this fact in \cite{Bore91} uses tools and basic (but sound) results from the theory of algebraic groups. However, it is possible to get it without the use of such machinery by means of inequalities of lists of integers. In particular using \eqref{codimorbit} and the following characterization for the inclusion $\overline\orb(M)\subseteq\overline\orb(L)$ (or, equivalently,
 $M\in\overline{\cal O}(L)$) in \cite[Th. 3]{Pokr86}. Here, we present the reformulation in \cite[Lemma 1.3]{Hoyo90}.

\begin{theorem}\label{dehoyos_th}
If $L$ and $M$ are matrix pencils of the same size and $h:=\rank L-\rank M\geq0$, then $M\in\overline\orb(L)$ if and only if the following three majorizations hold: {\rm (1)} $r(M) \prec_{ \rm w} r(L)+(h,h,\ldots)$; {\rm (2)}  $\ell(M) \prec_{ \rm w} \ell(L)+(h,h,\ldots)$; and {\rm (3)} $W(\lambda,L)\prec_{\rm w} W(\lambda,M)+(h,h,\ldots),$ for all $\lambda\in \overline{\CC}$.
\end{theorem}

We will need the following two lemmas on inequalities of non-negative integer lists.

\begin{lemma}\label{disdeltas_lem}
    If $k\geq1$, and $d_1,\hdots,d_k$ and $\delta_1,\hdots,\delta_k$ are integers such that
    {\rm(a)} $\delta_1\geq\delta_2\geq\cdots\geq\delta_k\geq0$, and {\rm(b)} $d_1+\cdots+d_j\geq0$, for all $j=1,\hdots,k$, then
    \begin{equation}\label{sum}
         d_1\delta_1+\cdots +d_k\delta_k\geq0.
    \end{equation}
     If 
     the equality holds in \eqref{sum}, then $d_1+\cdots+d_j=0$ or $\delta_j=\delta_{j+1}$, for $1\leq j\leq k$ (with $\delta_{k+1}:=0$).
\end{lemma}
\begin{proof} 
Using recursively conditions (a) and (b) in the statement, we get the following inequalities:
\begin{equation}\label{sum2}
\begin{array}{ccc}
 d_1\delta_1+d_2\delta_2+\cdots +d_k\delta_k&\geq&
 (d_1+d_2)\delta_2+d_3\delta_3+\cdots+d_k\delta_k\\&\geq &(d_1+d_2+d_3)\delta_3+\cdots+d_k\delta_k\\
 &\vdots&\vdots\\&\geq&(d_1+d_2+\cdots+d_k)\delta_k\geq0.
\end{array}
\end{equation}
For the claim about the equality, note that the equality in \eqref{sum} implies that all the inequalities in \eqref{sum2} become equalities. Then, subtracting in each of these equalities the right-hand side from the left one yields $(d_1+d_2+\dots+d_{j})(\delta_{j}-\delta_{j+1})= 0$, for all $j=1,2,\dots,k$.
\end{proof}

\begin{lemma}\label{main_lem}
Let $p\geq0$ and let $\alpha_1,\hdots,\alpha_k$ and $\beta_1,\hdots,\beta_k$ be integers such that
{\rm(i)} $0\leq\alpha_{i+1}\leq\alpha_{i}\leq p$ and $0\leq\beta_{i+1}\leq\beta_{i}\leq p$, for all $i=1,\hdots,k-1$, and {\rm(ii)} $\alpha_1+\cdots+\alpha_j\leq\beta_1+\cdots+\beta_j$, for all $j=1,\hdots,k$. Then
\begin{equation}\label{sumsquares}
    \sum_{i=1}^k\alpha_i^2\leq\sum_{i=1}^k\beta_i^2, \quad \mbox{and}
\end{equation}
    \begin{equation}\label{ineqproducts}
        p\alpha_1+\sum_{i=1}^{k-1}\alpha_i\alpha_{i+1}\leq p\beta_1+\sum_{i=1}^{k-1}\beta_i\beta_{i+1}.
    \end{equation}
Moreover, the equality is attained in \eqref{sumsquares} or \eqref{ineqproducts} if and only if $\alpha_i=\beta_i$, for all $i=1,\hdots,k$.
\end{lemma}
\begin{proof}
Note that \eqref{sumsquares} is equivalent to $\sum_{i=1}^k(\beta_i^2-\alpha_i^2)\geq0$, which is in turn equivalent to $\sum_{i=1}^k(\beta_i-\alpha_i)(\beta_i+\alpha_i)\geq0$, which follows from Lemma \ref{disdeltas_lem} for $d_i=\beta_i-\alpha_i$ and $\delta_i=\beta_i+\alpha_i$.

Next, let us assume that the equality in \eqref{sumsquares} holds. By Lemma \ref{disdeltas_lem}, for each $j=1,2,\dots,k$, $\alpha_j + \beta_j =\alpha_{j+1} + \beta_{j+1}$ (where $\alpha_{k+1}=\beta_{k+1}=0$) or $\sum_{i=1}^j (\beta_i - \alpha_i) = 0$. In addition, by the assumption (i), the equality $\alpha_{j}+\beta_j=\alpha_{j+1}+\beta_{j+1}$ is equivalent to $\alpha_j=\alpha_{j+1}$ and $\beta_{j}=\beta_{j+1}$. If $\sum_{i=1}^j (\beta_i - \alpha_i)>0$ for all $j=1,2,\dots,k$, then $\alpha_{j}+\beta_j=\alpha_{j+1}+\beta_{j+1}$ for $j=1,2,\dots,k$, which implies that $\alpha_j = \beta _j =0$ for $j=1,2,\dots,k$ and leads to a contradiction. Then, $\sum_{i=1}^{j} (\beta_i - \alpha_i) = 0$ for at least one $1\leq j\leq k$. Let $j_0 = \min \{j \, : \, \sum_{i=1}^{j} (\beta_i - \alpha_i) = 0 \}$ and note that $j_0 = 1$, because otherwise $\alpha_1 = \cdots = \alpha_{j_0}$ and $\beta_1 = \cdots = \beta_{j_0}$, which implies $j_0 (\beta_1 - \alpha_1) = \sum_{i=1}^{j_0} (\beta_i - \alpha_i) = 0$, namely $\beta_1 - \alpha_1 =0$, which is a contradiction. Thus $\beta_1 = \alpha_1$. Then, the sequences $\alpha_2 \geq \cdots \geq \alpha_k \geq 0$ and $\beta_2 \geq \cdots \geq \beta_k \geq 0$ satisfy the assumptions of the statement and the equality $\sum_{i=2}^k\alpha_i^2=\sum_{i=2}^k\beta_i^2$, so the argument above applied to these sequences yields  $\alpha_2=\beta_2$. Continuing in this way we get $\alpha_i=\beta_i$, for $i=1,2,\dots,k$.

Now let us prove \eqref{ineqproducts}.
Lemma \ref{disdeltas_lem} applied to $d_i=\beta_i-\alpha_i$ and $\delta_{i}=\alpha_{i-1}$, for $i=1,2,\dots,k$, with $\alpha_0:=p$, gives us $p(\beta_1-\alpha_1)+\sum_{i=1}^{k-1}\alpha_i(\beta_{i+1}-\alpha_{i+1})\geq 0$, which is equivalent to \begin{equation}\label{part_ineq1}
    p\alpha_1+\sum_{i=1}^{k-1}\alpha_i\alpha_{i+1}\leq p\beta_1+\sum_{i=1}^{k-1}\alpha_i\beta_{i+1}.
\end{equation}
Similarly, Lemma \ref{disdeltas_lem} with $d_i=\beta_i-\alpha_i$ and $\delta_i=\beta_{i+1}$, for $i=1,2,\dots,k-1$, yields
 $ p\beta_1+\sum_{i=1}^{k-1}\alpha_i\beta_{i+1}\leq p\beta_1+\sum_{i=1}^{k-1}\beta_i\beta_{i+1},$
which combined with \eqref{part_ineq1} gives \eqref{ineqproducts}.

Finally, if we assume that the equality is attained in \eqref{ineqproducts} then it is also attained in \eqref{part_ineq1}. Lemma \ref{disdeltas_lem} implies that, for each $j=1,2,\dots,k$,  $\alpha_1+\dots+\alpha_j=\beta_1+\dots+\beta_j$ or $\alpha_{j-1}=\alpha_{j}$, where $\alpha_0:=p$. If $\alpha_1<\beta_1$, then $\beta_1\leq p=\alpha_1$, which is a contradiction. Thus $\alpha_1 = \beta_1$. If there exists $\ell >1$ such that $\alpha_j=\beta_j$ and $\alpha_\ell<\beta_\ell$, for  $j=1,\dots,\ell-1<k$, then $\alpha_\ell=\alpha_{\ell-1}=\beta_{\ell-1}\geq \beta_\ell$, which is again a contradiction. Then, it must be $\alpha_i=\beta_i$, for $i=1,2,\dots,k$.
\end{proof}

\section{The main result}\label{main_sec}

We want to prove the following result.

\begin{theorem}\label{dim_th}
    Let $L$ and $M$ be two $m\times n$ matrix pencils with $\rank M\leq\rank L$, and set $h:=\rank L-\rank M$. If the following three majorization relationships are satisfied
\begin{itemize}
     \item[]\namedlabel{m1}{\rm(M1)} $r(M)\maj\, r(L)+(h,h,\hdots)$,
    \item[]\namedlabel{m2}{\rm(M2)} $\ell(M)\maj\,\ell(L)+(h,h,\hdots)$, and  
    \item[]\namedlabel{m3}{\rm(M3)} $W(\mu,L)\maj\, W(\mu,M)+(h,h,\hdots)$, for all $\mu\in\overline\CC$,
\end{itemize}
    then
    \begin{equation}\label{codim_ineq}
    \codim \orb(L)\leq\codim\orb(M),
\end{equation}
    where $\codim\orb(N)$ denotes the codimension of the orbit of $N$, given by \eqref{codimorbit}. Moreover, the equality holds in \eqref{codim_ineq} if and only if $h=0$ and ``$\maj$" is replaced by ``$=$" in \ref{m1}--\ref{m3}.
\end{theorem}

A key tool in our proof is a characterization of the inclusion relationship between orbit closures given in Theorem \ref{dehoyos_th} (namely, a characterization of \ref{m1}--\ref{m3} in Theorem \ref{dim_th}) by means of six elementary transformations. This result was derived in \cite{Bole98} following \cite{Pokr86}. See also \cite{sergeichuk-laa2021} for a complete modern treatment. Here we follow \cite[Th. 2.2]{DmDo18}, where the notation $A\rightsquigarrow B$ means that the pencil $A$ is replaced by the pencil $B$ and $J_0 (\mu)$ is the empty matrix.

\begin{theorem}\label{rules_th}
    Let $L$ and $M$ be two matrix pencils as in the statement of Theorem {\rm\ref{dim_th}}. Then \ref{m1}--\ref{m3} in Theorem {\rm\ref{dim_th}} hold if and only if the KCF of $L$ can be obtained from the one of $M$ after a sequence of changes, where each change can be of the following six types:
    \begin{enumerate}
        \item\label{rule1} $L_{j-1}\oplus L_{k+1}\rightsquigarrow L_j\oplus L_k$, with $1\leq j\leq k$.
        \item $L_{j-1}^\top\oplus L_{k+1}^\top\rightsquigarrow L_j^\top\oplus L_k^\top$, with $1\leq j\leq k$.
        \item $L_j\oplus J_{k+1}(\mu)\rightsquigarrow L_{j+1}\oplus J_k(\mu)$, with $j,k\geq0$ and $\mu\in\overline\CC$.
        \item $L_j^\top\oplus J_{k+1}(\mu)\rightsquigarrow L_{j+1}^\top\oplus J_k(\mu)$, with $j,k\geq0$ and $\mu\in\overline\CC$.
        \item\label{rule5} $J_j(\mu)\oplus J_k(\mu)\rightsquigarrow J_{j-1}(\mu)\oplus J_{k+1}(\mu)$, with $1\leq j\leq k$ and $\mu\in\overline\CC$.
        \item $L_p\oplus L^\top_q\rightsquigarrow \bigoplus_{i=1}^s J_{n_i}(\mu_i)$, with $p+q+1=\sum_{i=1}^sn_i$,  $\mu_i\in\overline\CC$, and $\mu_i\neq\mu_{i'}$ for $i\neq i'$.
    \end{enumerate}
\end{theorem}

Now, we are in a position to prove Theorem \ref{dim_th}.

\medskip

\noindent{\em Proof of Theorem {\rm\ref{dim_th}}.} By \eqref{codimorbit},  the inequality \eqref{codim_ineq} is equivalent to
\begin{equation}\label{dimineq2}
\begin{array}{c}
     \displaystyle\ell_0(M)n+r_0(M)m-\sum_{i\geq0} r_i(M)r_{i+1}(M)-\sum_{i\geq0}\ell_i(M)\ell_{i+1}(M)
     +\sum_{\lambda}\sum_{i\geq1} W_i(\la,M)^2\\
     \geq\displaystyle\ell_0(L)n+r_0(L)m-\sum_{i\geq0} r_i(L)r_{i+1}(L)-\sum_{i\geq0}\ell_i(L)\ell_{i+1}(L)
     +\sum_{\lambda}\sum_{i\geq1} W_i(\la,L)^2.
\end{array}
\end{equation}
Using that $r_0(M)-r_0(L)=\ell_0(M)-\ell_0(L)=h$, 
which follows from \eqref{nrank}, we get
$
\ell_0(M)n+r_0(M)m-\ell_0(L)n-r_0(L)m =h(m+n).
$
Hence, \eqref{dimineq2} (and, so, \eqref{codim_ineq}) is equivalent to
\begin{equation}\label{dimineq3}
\begin{array}{c}
    \displaystyle\sum_{i\geq0} r_i(L)r_{i+1}(L)-\sum_{i\geq0} r_i(M)r_{i+1}(M)+\sum_{i\geq0} \ell_i(L)\ell_{i+1}(L)-\sum_{i\geq0}\ell_i(M)\ell_{i+1}(M)\\\displaystyle
    +\sum_{\la}\sum_{i\geq1} W_i(\la,M)^2-\sum_{\la}\sum_{i\geq1} W_i(\la,L)^2+h(m+n)\geq0.
    \end{array}
\end{equation}

We first prove \eqref{dimineq3} when $h=0$.
Note that, if we set $\alpha_i:=W_i(\la,L)$ and $\beta_i:=W_i(\la,M)$, for $i\geq1$, then the definition of the Weyr characteristic guarantees that $\alpha_i,\beta_i$ satisfy (i) in Lemma \ref{main_lem}, whereas condition \ref{m3} guarantees that (ii) in that lemma is also satisfied. Hence, \eqref{sumsquares} implies
\begin{equation}\label{squares-ineq}
\sum_{\la}\sum_{i\geq1} W_i(\la,L)^2\leq \sum_{\la}\sum_{i\geq1} W_i(\la,M)^2.
\end{equation}
Similarly, if either $\alpha_i:=r_i(M),\beta_i:=r_i(L)$, for $i\geq1$, together with $p:=r_0(L)=r_0(M)$ (since $h=0$), or $\alpha_i:=\ell_i(M),\beta_i:=\ell_i(L)$, for $i\geq1$, together with $p:=\ell_0(L)=\ell_0(M)$ (again because $h=0$), then (M1)-(M2) guarantee that $\alpha_i,\beta_i,p$ satisfy (i)--(ii) in Lemma \ref{main_lem}, so \eqref{ineqproducts} gives
\begin{equation}\label{mineq}
\sum_{i\geq0} r_i(M)r_{i+1}(M)\leq \sum_{i\geq0} r_i(L)r_{i+1}(L)\quad\mbox{and}\quad
\sum_{i\geq0}\ell_i(M)\ell_{i+1}(M)\leq \sum_{i\geq0}\ell_i(L)\ell_{i+1}(L).
\end{equation}
Adding up \eqref{squares-ineq} and \eqref{mineq}  we obtain \eqref{dimineq3} for $h=0$.

For the case $h>0$, we are going to prove \eqref{dimineq3} for $L=N\oplus N_L$ and $M=N\oplus N_M$, with $N_L$ and $N_M$ being the pencils in, respectively, the right-hand side and the left-hand side of each rule 1-6 in the statement of Theorem \ref{rules_th}. Note that $h=0$ in rules 1-5, and we have already proved that \eqref{dimineq3} holds in this case. Hence, it only remains to prove that it holds for rule 6, namely for
$$
M=N\oplus\ L_p\oplus L_q^T\rightsquigarrow L=N\oplus\bigoplus_{i=1}^s J_{n_i}(\la_i),\qquad \sum_{i=1}^s n_i=p+q+1,\qquad \la_i\neq\la_j,\ i\neq j,
$$
for some pencil $N$. Then,
$
r(M)=r(L)+(\underbrace{1,\hdots,1}_{p+1},0, \hdots, 0),\
    \ell(M)=\ell(L)+(\underbrace{1,\hdots,1}_{q+1},0, \hdots, 0),\label{llm}
$
and
\begin{equation}
    W(\la_i,L)=W(\la_i,M)+(\underbrace{1,\hdots,1}_{n_i}, 0, \hdots, 0) \label{wlm},\qquad i=1,\hdots,s,
\end{equation}
so, in particular, $h=r_0(M)-r_0(L)=1$.  Note that \eqref{dimineq3} is equivalent to
\begin{equation}\label{dimineq4}
    \begin{array}{c}
    \displaystyle\sum_{i\geq0} r_i(L)r_{i+1}(L)+\sum_{i\geq0} \ell_i(L)\ell_{i+1}(L)
    +\sum_{\la}\sum_{i\geq1} W_i(\la,M)^2+h(m+n)\\
    \displaystyle\geq
    \sum_{i\geq0} r_i(M)r_{i+1}(M)+\sum_{i\geq0}\ell_i(M)\ell_{i+1}(M)+\sum_{\la}\sum_{i\geq1} W_i(\la,L)^2.
    \end{array}
\end{equation}
Starting from the left-hand side in \eqref{dimineq4}, with $h=1$, we get:\allowbreak
$$\allowbreak
\begin{array}{c}
\displaystyle\sum_{i\geq 0}r_i(L)r_{i+1}(L)+\sum_{i\geq0}\ell_i(L)\ell_{i+1}(L)+\sum_{\la}\sum_{i\geq1}W_i(\la,M)^2+(m+n)\\
=\displaystyle\sum_{i=0}^{p-1}(r_i(M)-1)(r_{i+1}(M)-1)+(r_p(M)-1)r_{p+1}(M)+\sum_{i\geq p+1}r_i(M)r_{i+1}(M)\\
+\displaystyle\sum_{i=0}^{q-1}(\ell_i(M)-1)(\ell_{i+1}(M)-1)+(\ell_q(M)-1)\ell_{q+1}(M)+\sum_{i\geq q+1}\ell_i(M)\ell_{i+1}(M)\\
+\displaystyle\sum_{\la\neq\la_1,\hdots,\la_s}\sum_{i\geq1} W_i(\la,L)^2+\sum_{k=1}^s\sum_{i=1}^{n_k}(W_i(\la_k,L)-1)^2
+\sum_{k=1}^s\sum_{i\geq n_k+1}W_i(\la_k,L)^2+m+n\\
=\displaystyle\sum_{i\geq0}r_i(M)r_{i+1}(M)-\sum_{i=0}^{p-1}(r_i(M)+r_{i+1}(M))+p-r_{p+1}(M)\\
+\displaystyle\sum_{i\geq0}\ell_i(M)\ell_{i+1}(M)-\sum_{i=0}^{q-1}(\ell_i(M)+\ell_{i+1}(M))+q-\ell_{q+1}(M)\\
+\displaystyle\sum_{\la\neq\la_1,\hdots,\la_s}\sum_{i\geq1}W_i(\la,L)^2+\sum_{k=1}^s\sum_{i=1}^{n_k}W_i(\la_k,L)^2-2\sum_{k=1}^s\sum_{i=1}^{n_k}W_i(\la_k,L)\\
+\displaystyle\sum_{k=1}^sn_k+\sum_{k=1}^s\sum_{i\geq n_k+1}W_i(\la_k,L)^2+m+n\\
\end{array}
$$
$$
\begin{array}{c}
=\displaystyle \sum_{i\geq0}r_i(M)r_{i+1}(M)+\sum_{i\geq0}\ell_i(M)\ell_{i+1}(M)+\sum_{\la}\sum_{i\geq1}W_i(\la,L)^2\\-\displaystyle2\sum_{i=1}^{p-1}r_i(M)-r_0(M)-r_p(M)-r_{p+1}(M)+p\\
-\displaystyle2\sum_{i=1}^{q-1}\ell_i(M)-\ell_0(M)-\ell_q(M)-\ell_{q+1}(M)+q
\displaystyle -2\sum_{k=1}^s\sum_{i=1}^{n_k}W_i(\la_k,L)+\sum_{k=1}^sn_k+m+n.
\end{array}
$$
Hence, it remains to show that
\begin{equation}\label{remain}
\begin{array}{c}
\displaystyle m+n+p+q+\sum_{k=1}^sn_k-2\left(\sum_{i=1}^{p-1}r_i(M)+
\sum_{i=1}^{q-1}\ell_i(M)+ \sum_{k=1}^s\sum_{i=1}^{n_k}W_i(\la_k,L)\right)\\
-r_0(M)-r_p(M)-r_{p+1}(M)-\ell_0(M)-\ell_q(M)-\ell_{q+1}(M)\geq0.
\end{array}
\end{equation}
Adding up \eqref{m+p} and \eqref{n+q}, we get
\begin{equation}\label{m+n}
m+n=2\left(\sum_{i\geq1}r_i(M)+\sum_{i\geq1}\ell_i(M)+\sum_\la\sum_{i\geq1}W_i(\la,M)\right)+r_0(M)+\ell_0(M).
\end{equation}
From \eqref{wlm} 
we get the following identity:
\begin{equation}\label{regularpart}
   \sum_\la\sum_{i\geq1}W_i(\la,M)+\sum_{k=1}^sn_k= \sum_{\la}\sum_{i\geq1}W_i(\la,L)\geq \sum_{k=1}^s\sum_{i=1}^{n_k}W_i(\la_k,L).
\end{equation}
From \eqref{m+n} and the identity $\sum_{k=1}^sn_k=p+q+1$, the left-hand side of \eqref{remain} is equal to
\begin{equation}\label{geq1}
\begin{array}{c}
\displaystyle2\left(\sum_{k=1}^sn_k+\sum_{\la}\sum_{i\geq1}W_i(\la,M)-\sum_{k=1}^s\sum_{i=1}^{n_k}W_i(\la_k,L)+\sum_{i\geq p+2}r_i(M)+\sum_{i\geq q+2}\ell_i(M)\right)\\
+r_p(M)+r_{p+1}(M)+\ell_q(M)+\ell_{q+1}(M)-1.
\end{array}
\end{equation}
By \eqref{regularpart} and $r_p(M)\geq1$, $\ell_q(M)\geq1$, we conclude that \eqref{geq1} is at least $1$, so \eqref{remain} follows.

Now, let us prove the claim in the statement regarding the equality in \eqref{codim_ineq}. First, we note that $h=0$ for rules \ref{rule1}--\ref{rule5} in the statement of Theorem \ref{rules_th}, so the equality in \eqref{codim_ineq} implies that the equality is attained in \eqref{dimineq3} (for $h=0$), and then it is also attained in \eqref{squares-ineq} and \eqref{mineq}. By Lemma \ref{main_lem}, this implies $W(\la,L)=W(\la,M)$, for all $\la\in\overline\CC$, together with $r(L)=r(M)$ and $\ell(L)=\ell(M)$.

 For rule 6 in Theorem \ref{rules_th}, we have proved that the expression in \eqref{geq1} is at least $1$, which implies that $\codim\orb(M)-\codim\orb(L)\geq1$ each time that rule 6 is applied. Thus, if the equality in \eqref{codim_ineq} is attained, only rules \ref{rule1}--\ref{rule5} are allowed in going from $M$ to $L$, and in this case $h=0$ and the identity holds in the majorizations \ref{m1}--\ref{m3} in the statement.
$\square$

We want to emphasize that setting $h=0$ and replacing ``$\maj$" by ``$=$" in \ref{m1}--\ref{m3} in the statement of Theorem \ref{dim_th} is equivalent to say that $M\in\orb(L)$.


An alternative approach to prove Theorem \ref{dim_th} is by using the formulas for the codimension in terms of the Segre characteristic in \cite{DeEd95}. Following a similar approach to the one in this note, this would require to analyze independently all six changes of the eigenstructure described in Theorem \ref{rules_th}, and to prove that in every single change the codimension inequality is satisfied.

\bigskip

\noindent{\bf Acknowledgments}.  The authors thank Andrii Dmytryshyn for suggesting the use of Theorem \ref{rules_th} to prove the main result. This work is part of grant PID2023-147366NB-I00 funded by MICIU/AEI/ 10.13039/501100011033 and FEDER/UE. Also funded by RED2022-134176-T and by the program Excellence Initiative -
Research University at the Jagiellonian University in Kraków.

\bibliographystyle{elsarticle-num}

\end{document}